\documentclass[preprint,12pt]{article}
\usepackage[utf8]{inputenc}
\usepackage{amsmath}
\usepackage{amsthm}
\usepackage{amsfonts}
\usepackage[margin=1in]{geometry}
\usepackage{amssymb}
\usepackage{array}
\usepackage{tikz}
\usepackage{float}
\usepackage{xcolor}
\usepackage{graphicx}
\usepackage{framed}
\usepackage{todonotes}
\usepackage{algorithm}
\usepackage{algpseudocode}
\usepackage{listings}
\usepackage{color}

\definecolor{mygreen}{RGB}{28,172,0} 
\definecolor{mylilas}{RGB}{170,55,241}

\newtheorem{theorem}{Theorem}[section]
\newtheorem{remark}{Remark}[section]
\newtheorem{lemma}[theorem]{Lemma}

\newtheorem{definition}[theorem]{Definition}
\newtheorem{assumption}[theorem]{Assumption}
\newtheorem{example}[theorem]{Example}

\newcommand{\bb}{\mathbf{b}}
\newcommand{\ff}{\mathbf{f}}
\newcommand{\bgg}{\mathbf{g}}
\newcommand{\pp}{\mathbf{p}}
\newcommand{\rr}{\mathbf{r}}
\newcommand{\uu}{\mathbf{u}}
\newcommand{\vv}{\mathbf{v}}
\newcommand{\ww}{\mathbf{w}}
\newcommand{\xx}{\mathbf{x}}
\newcommand{\yy}{\mathbf{y}}
\newcommand{\ttt}{{\eta}}
\newcommand{\aas}{\mathbf{A}_{\Omega_s}}

\newcommand{\schur}{{ (1+\ttt)^2 / C_2^2}}

\newcommand{\chen}[1]{\textcolor{black}{#1}}
\newcommand{\chennew}[1]{\textcolor{black}{#1}}
\newcommand{\chennewnew}[1]{\textcolor{black}{#1}}

\title{GMRES Convergence Analysis for Nonsymmetric Saddle-Point Systems 
When the Field of Values Contains the Origin}

\author{
  \textsc{Hao Chen}\footnotemark[2]%
  \and
  \textsc{Chen Greif}\footnotemark[2]\hspace{4pt}\footnotemark[1]%
}

\date{
  \em Dedicated to Daniel Szyld on the occasion of his 70th birthday \\[6pt]
  To be published in Linear Algebra and Its Applications
}

\begin{document}
\maketitle

\begingroup
\renewcommand{\thefootnote}{\fnsymbol{footnote}} 
\footnotetext[2]{Department of Computer Science, The University of British Columbia, Vancouver, Canada. Email: \{haochen3,greif\}@cs.ubc.ca.}
\footnotetext[1]{The work of this author was supported in part by the Natural Sciences and Engineering Research Council of Canada. }
\endgroup

\maketitle

\begin{abstract}
\noindent We present a field-of-values (FOV) analysis for preconditioned nonsymmetric saddle-point linear systems, where zero is included in the field of values of the matrix. We rely on recent results of Crouzeix and Greenbaum [Spectral sets: numerical range and beyond. SIAM Journal on Matrix Analysis and Applications, 40(3):1087-\chen{1101}, 2019], showing that a convex region with a circular hole is a spectral set. Sufficient conditions are derived for convergence independent of the matrix dimensions. We apply our results to preconditioned nonsymmetric saddle-point systems, and show their applicability to families of block preconditioners that have not been previously covered by existing FOV analysis. A limitation of our theory is that the preconditioned matrix is required to have a small skew-symmetric part in norm. Consequently, our analysis may not be applicable, for example, to fluid flow problems characterized by a small viscosity coefficient. Some numerical results illustrate our findings.
\end{abstract}

{\bf Keywords.} field of values; nonsymmetric saddle-point systems; GMRES convergence; block preconditioner

\section{Introduction}
The field of values of a matrix is an indispensable tool in linear algebra and its applications.  It is defined as follows.
\begin{definition}
  \label{def:fov}
  Given a matrix $A \in \mathbb{C}^{n \times n}$, the field of values (FOV)  of  $A$ is defined as
  $$
  W(A) = \left\{ \frac{\xx^* A \xx}{\xx^* \xx} : \quad \xx \in \mathbb{C}^n \chennew{\backslash \{\mathbf 0\}}  \right\}
  $$
  and the $H$-field of values  of  $A$, given a \chen{Hermitian positive definite} matrix $H \in \mathbb{C}^{n \times n}$, is defined as
  $$
  W_H(A) = \left\{ \frac{\xx^* H A \xx}{\xx^* H \xx} : \quad \xx \in \mathbb{C}^n \chennew{\backslash \{\mathbf 0\}} \right\}.
  $$
\end{definition}

Early work on the topic was published in \cite{eiermann1993fields,klawonn1999block} and in several other papers; see \cite{benzi2021some} for a recent expository paper that provides an overview of the use of \chennew{the} FOV, its history and development, and a comprehensive list of references. 

Throughout this paper, we extensively use the notion of a weighted norm, which we define as follows.
\begin{definition} 
  Given a Hermitian positive definite  matrix $H\in \mathbb{C}^{n\times n}$, the $H$-norm of a vector $\uu\in \mathbb{C}^n$ is defined as 
\begin{equation*}
\| \uu\|_{H} = (\uu,H \uu)^{1/2},
\end{equation*}  
\chennewnew{
and the corresponding $H$-norm of a matrix $A  \in \mathbb{C}^{n \times n}$ is the induced norm 
\begin{equation*}
\| A\|_{H} = \max_{ \uu \in \mathbb{C}^n \chennewnew{\backslash \{\mathbf 0\}} } \frac{\| A \uu \|_H}{\| \uu \|_H}.
\end{equation*} 
If $A$ is nonsingular, its  $H$-condition number  is defined as
\[ 
\kappa_H(A) = \|A\|_H \|A^{-1}\|_H.
\]
 }
\end{definition}

In the context of this work, we are interested in the use of FOV to establish the scalability of Krylov subspace iterative solvers (specifically, GMRES \cite{saad1986gmres}) for large and sparse nonsymmetric saddle-point systems:
\begin{equation}
\label{eq:K_sys}
 \begin{bmatrix}
    F & B^T \\
    B & 0
  \end{bmatrix}
  \begin{bmatrix}
      \uu \\ \pp
  \end{bmatrix} = 
  \begin{bmatrix}
      \ff \\ \bgg
  \end{bmatrix} ,
\end{equation}
where $F \in \mathbb{R}^{n \times n}$ is nonsymmetric, $B \in \mathbb{R}^{m \times n}$ has full row rank, and $\uu, \ff \in {\mathbb R}^n, \ \pp, \bgg \in {\mathbb R}^m.$

Significant work has been done on this topic 
\cite{beik2022preconditioning, chidyagwai2016constraint,  klawonn1999block, loghin2004analysis, ma2016robust}, but to the best of our knowledge, the analysis is limited to the situation where $0$ is not included in the field of values. 
Our goal is to perform an FOV analysis for preconditioned saddle-point systems in the case where the origin is included.

Part of our motivation in considering the field of values is that spectral analysis may be limited for this family of linear systems: for nonsymmetric saddle-point systems arising from partial differential equations, the condition number of the eigenvector matrix of the preconditioned matrix typically increases with the matrix dimensions. \chennewnew{Let $\mathbf{P}_j$ denote all polynomials $p$ of degree $\le j$ that satisfy $p(0)=1$.}
\chen{Then, considering solving \eqref{eq:K_sys} using GMRES in the $H$\chennewnew{-}norm, if $\rr_k$ denotes the residual of the $k$th iteration,
the inequality
    \begin{equation}
  \frac{\|\rr_j\|_{H}}{\|\rr_0\|_{H}}
  \;\le\;
  \kappa_{H}(V)\, \min_{p\in \mathbf{P}_j} \max_{k=1,\dots,n+m} \bigl|p(\lambda_k)\bigr|,
  \label{eq:spec}
\end{equation}
where $V$ is the best \chennewnew{$H$-conditioned} matrix of eigenvectors of the saddle-point matrix and $\{ \lambda_k\}$ are its $n+m$ eigenvalues,  may not capture the possibility of the iteration counts being independent or nearly independent of the matrix dimensions.
In the context of the Navier-Stokes equations, for example, this happens even for a large viscosity coefficient (see Remark \ref{rem:smallskew}). Thus, an analysis of the eigenvalues of the preconditioned matrix is often insufficient to theoretically prove scalability (in situations where it is expected) of an iterative method for such linear systems. }

Following the terminology of \cite[Eq. (1)]{crouzeix2019spectral},
while restricting our attention to discrete linear operators, polynomials, and the $H$-norm,  we say that for a closed subset $X \subset {\mathbb C}$ \chen{and} a matrix $A$,  $X$ is a $K$-spectral set for $A$ if for any polynomial $p$
$$ \|p(A)\|_H \leq K \sup_{z \in X} |p(z)|.$$

\begin{theorem}\cite[Theorem 6]{crouzeix2017numerical}
  \label{thm:fov_std}
  Let $A$ be a matrix of the same dimensions as $H$. Then, $W_H(A)$ is a $(1+\sqrt{2})$-spectral set for $A$.
\end{theorem}

In the sequel, we will be using GMRES with respect to $H$-norm, or equivalently  the $H$-weighted inner product $\langle \cdot,\cdot \rangle_H$. 
Applying Theorem \ref{thm:fov_std}, we obtain the following convergence bound.
\begin{theorem}\cite{crouzeix2019spectral} \label{thm:pol}
  Let $\rr_k=\bb-A \xx_k$ be the residual of the $k$-th iteration, $\xx_k$, of GMRES with respect to the $H$-norm applied to the linear system $A \xx = \bb$,  
  and let $\mathbf{P}_j$ denote all polynomials $p$ of degree $\le j$ that satisfy $p(0)=1$. Then,
  \begin{equation*}
    \frac{\| \rr_j\|_H} {\| \rr_0\|_H} \leq \min_{p \in \mathbf{P}_j} \| p(A) \|_H \leq (1+\sqrt{2}) \min_{p\in \mathbf{P}_j} \max_{z \in W_H(A)}| p(z) |.
  \end{equation*}
\end{theorem}
A challenge is that when $0\in W_H(A)$, we have $\min_{p\in \mathbf{P}_j} \max_{z \in W_H(A)}| p(z) | = 1$, and Theorem \ref{thm:pol}  fails to provide a useful bound on GMRES convergence in this case. The presence of a zero in the field of values is, in fact, common in saddle-point systems: the (2,2)-block of a saddle-point system preconditioned with a block-diagonal matrix can be $0$. 
Recently, Crouzeix and Greenbaum \cite{crouzeix2019spectral} \chennew{defined} a convex region with a circular hole and showed that it is a spectral set. This can be  used to analyze cases when zero is included in the field of values. 

\chen{In \cite{crouzeix2019spectral} it is shown that if $\Omega_{CG}$ represents a domain constructed as $W_H(A)$  with a \chennew{ disk about the origin} removed  that has radius $1/w$, where $w$ denotes the numerical radius of $A^{-1}$, then $\Omega_{CG}$ is a $(3+2\sqrt{3})$-spectral set for $A$.  This bound can be improved to $2 + \sqrt{7}$ if a 
\chennew{smaller disk about the origin} of radius $1/\|A^{-1}\|$ is removed from $W_H(A)$; see \cite[Theorem 2]{crouzeix2019spectral} and  \cite{GreenbaumWellen2024}, which further refines the discussion of available bounds for a few cases of interest. } 
\begin{theorem} 
  \cite{crouzeix2019spectral,  GreenbaumWellen2024}
  \label{thm:fieldofvalues}
  Let $A$ be a matrix of the same dimensions as $H$. Then, $\Omega_{CG} = W_H(A) \cap \{z\in C: |z| \geq \| A^{-1}\|_{H}^{-1}\}$ is a \chen{$(2+\sqrt{7})$-spectral set for $A$}.
\end{theorem}

In \cite{embree2022descriptive}, the author 
presents a simple example to illustrate the potential of this result in the context of convergence of GMRES. 

 \chen{In this paper, we consider a special family of saddle-point linear systems that arise from discretization of fluid flow problems.} We include in our discussion block-diagonal preconditioners and certain block-triangular preconditioners for which no previous FOV analysis is available.
On the other hand, our analysis has some limitations compared to the well-studied FOV analysis that excludes the origin. For example, in \cite{loghin2004analysis}, scaling is effectively used to allow for applying FOV analysis to the discrete Navier-Stokes equations with a small viscosity coefficient when the field of values does not include the origin. In our analysis we are not able to utilize scalings in the same manner, and we require the skew-symmetric part of the linear system to be small norm-wise.

The remainder of this paper is structured as follows. In Section \ref{sec:main} we \chen{review some useful results in the literature and} present an analysis that deals with zero in the field of values. In Section \ref{sec:fov_sp} we specialize our results to saddle-point systems. In Section~\ref{sec:experiments} we discuss a few examples of relevant applications and present some numerical results.  Finally, we draw some conclusions in Section \ref{sec:conclusions}.

\section{FOV Analysis that Includes Zero}
\label{sec:main}
In this section, we derive sufficient conditions that will serve us in our analysis for saddle-point systems. 
\subsection{Preliminaries}
Let us present a few known results that we will use in our analysis.  \chennewnew{Some of the definitions and results that follow are specialized to real matrices and vectors. } 
\begin{definition} \chennewnew{\cite[Page 311]{HornJohnson1985}}
For two symmetric positive definite matrices $H_1 \in \mathbb{R}^{n\times n}$ and $H_2 \in \mathbb{R}^{m\times m}$, we define the $ (H_1,H_2)$-norm for a matrix  $M\in \mathbb{R}^{m\times n}$ as 
\begin{equation*}
\| M\|_{H_1, H_2} = \max_{\chennewnew{\vv \in \mathbb{R} \backslash \{\mathbf 0\}} }\frac{\| M \vv\|_{H_2}}{\| \vv\|_{H_1}}.
\end{equation*}
\label{def:H1H2}
\end{definition}
The following equalities, given in \cite[Eq. (2.4)]{loghin2004analysis}, are immediate from Definition \ref{def:H1H2}: $$ 
    \| H_2^{-1/2} M H_1^{-1/2}\|_2 = \| M\|_{H_1,H_2^{-1}} = \| MH_1 ^{-1}\|_{H_1^{-1},H_2^{-1}} = \|H_2^{-1}M\|_{H_1, H_2}.
 $$
 The following \chen{additional} properties from \cite{loghin2004analysis} are useful for our analysis.
\begin{lemma}[{\cite[Lemma \chen{2.1}]{loghin2004analysis}}]
\label{lem:norms}
  Let $M \in \mathbb{R}^{m\times n}$ have full rank, and let $H_1 \in \mathbb{R}^{n\times n}$, $H_2 \in \mathbb{R}^{m\times m}$ be two symmetric  positive definite matrices. Then
  \begin{itemize}
  \item[(i)]
  $ \displaystyle
    \| M\|_{H_1, H_2^{-1}} = \max_{\vv\in \mathbb{R}^n \backslash \{\mathbf 0\} }\max_{\ww\in \mathbb{R}^m \backslash \{\mathbf 0\} } \frac{\ww^T M \vv}{\| \vv\|_{H_1} \| \ww\|_{H_2}}
$. 
 \item[(ii)] If $m=n$,
  \begin{equation*}
    \| M^{-1}\|_{H_2^{-1},H_1}^{-1} = \min_{\vv\in \mathbb{R}^n \backslash \{\mathbf 0\} }\max_{\ww\in \mathbb{R}^m \backslash \{\mathbf 0\} } \frac{\ww^T M \vv}{\| \vv\|_{H_1} \| \ww\|_{H_2}} .
  \end{equation*}
  \item[(iii)]  
  If $H_i\in \mathbb{R}^{n_i \times n_i},i=1,2,3$ are three symmetric and positive definite matrices and $R \in \mathbb{R}^{n_1 \times n_2}, Q\in \mathbb{R}^{n_2 \times n_3}$  then
  \begin{equation*}
   \| R Q\|_{H_3, H_1} \leq \| Q\|_{H_3,H_2} \| R\|_{H_2, H_1}.
  \end{equation*}
  \end{itemize}
\end{lemma}

The following result from \cite{driscoll1998potential}, adapted to our notation and context, is useful in our analysis.

\begin{theorem}[{\cite[Theorem 1]{driscoll1998potential}}]\label{thm:exp_decay}
    Let \chennew{$\mathbf{P}_n$} denote the set of polynomials $p$ of degree at most $n$  with $p(0) = 1$. For a compact set $S$ in the complex plane, with the origin not included in or surrounded by $S$  and no isolated points, define 
    $$
    E_n(S) = \min_{p\in \chennew{\mathbf{P}_n}}\max_{z\in S}|p(z)|
    $$
    and the corresponding {\underline  {estimated {asymptotic} convergence factor}} $$
    \rho = \lim_{n \to \infty }(E_n(S))^{1/n}.
    $$
    Let $g(z)$ be the Green's function associated with $S$, defined in the exterior of $S$, satisfying $\nabla^2 g = 0$ outside of $S$, $g(z)\to 0$ as $z \to \partial S$, and $g(z)-\log|z|\to C$ as $|z| \to \infty$ for some \chen{constant} $C$. Then,
    $$ \rho = \exp(-g(0)).$$
\end{theorem}

\subsection{Sufficient Conditions}

\begin{lemma}\label{lem:fov_main} 
  \chen{Given constants  $a,b,c>0$ with
  \begin{equation} 
  bc<1,
  \label{cond4}
  \end{equation}
  and a convergence tolerance $\varepsilon>0$, for any nonsingular $n \times n$ matrix $A$  and positive definite $H$ of the same dimensions that satisfy 
   \begin{subequations}
   \begin{align}
       \| A \|_{H} \leq a;  \label{cond1}\\
  \| A ^{-1} \|_H \leq b  ; \label{cond2} \\ 
  \|( HA - A^T H ) / 2\|_{H,H^{-1}} \leq c \label{cond3}
    \end{align}
    \end{subequations}
  for these $a$, $b$, and $c$ values, 
  and any $n$-vector $\ff$ and an initial guess $\xx_0$, 
  there exists some integer $m \ge 1$
  that depends on $a$, $b$, and $c$, such that
  the residual $\rr_m = \ff - A \xx_m$ generated by $m$ steps of GMRES applied to the linear system $A \xx=\ff$ in the $H$-norm satisfies $$\|\rr_m\|_H/\|\rr_0\|_H \le \varepsilon.$$}%
\end{lemma} 
\begin{proof}
  We first derive a bound on the field of values of $A$.
  Suppose the conditions hold. Then, for any $z \in W_H(A)$, we have $|z| \leq \| A\|_H \leq a$ and 
  \begin{align*}
  |Im(z)| &\leq \max_{\xx\in \mathbb{C}^{n}}\left| \left( \frac{\xx^* H A \xx}{\xx^* H \xx} - \left(\frac{\xx^* H A \xx}{\xx^* H \xx}\right)^* \right) \Big/2\right| \\
    &= \max_{\xx\in \mathbb{C}^{n}}\left| \left( \frac{\xx^* ( H A - A^T H) \xx}{ 2 \xx^* H \xx} \right)\right| \\
    &\leq \|( HA - A^T H ) / 2\|_{H,H^{-1}} \leq c.
  \end{align*}
\chennewnew{Notice that  the last inequality holds because $H A - A^T H$ is skew-symmetric, and hence its field of values lies on the imaginary axis \cite[Property 1.2.5]{HornJohnson1991}.}

We then have  $$\Omega_{CG} \subseteq \Omega_{D} :=  \{z:  \frac{1}{b}\leq |z| \leq a\} \cap \{ z \in \mathbb{C}: |Im(z)| \leq c\}.$$
By Theorem \ref{thm:fieldofvalues}, we have the  GMRES convergence result
\begin{align*}
  \frac{\| \rr_j\|_H} {\| \rr_0\|_H} &\leq \min_{p \in \mathbf{P}_j} \| p(A) \|_H \leq \chen{(2+\sqrt{7})} \min_{p \in \mathbf{P}_j} \max_{z \in \Omega_{CG}}| p(z) |. \\
\end{align*}
\begin{figure}[htbp]
  \centering
\begin{tikzpicture}
  \begin{scope}
    \clip (-2, -0.8) rectangle (2, 0.8);
    \fill[cyan, opacity=0.3] (0, 0) circle (1.5);
    \fill[white] (0, 0) circle (1);
  \end{scope}
  
  \draw[->] (-2, 0) -- (2, 0) node[right] {$\text{Re}(z)$};
  \draw[->] (0, -1.8) -- (0, 1.8) node[above] {$\text{Im}(z)$};
  
  \draw[black] (0, 0) circle (1.5);
  
  \draw[black, thick, dashed] (0, 0) circle (1);
  
  \draw[black, thick, dotted] (-1.5, 0.8) -- (1.5, 0.8);
  \draw[black, thick, dotted] (-1.5, -0.8) -- (1.5, -0.8);
  
  \node at (1.8, 0.8) {$c$};
  \node at (1.8, -0.8) {$-c$};
  
  \draw[<-, dashed, thick] (45:1.5) -- (45:2.2);
  \node at (45:2.4) {$a$};
  
  \draw[<-, dashed, thick] (135:1) -- (135:1.8);
  \node at (135:2.0) {$\frac{1}{b}$};
  
\end{tikzpicture}
\caption{The shaded region is $\Omega_{D}$ when conditions \eqref{cond4} and \eqref{cond1}--\eqref{cond3} of Lemma \ref{lem:fov_main} hold}
\label{fig:aclt1}
\end{figure}
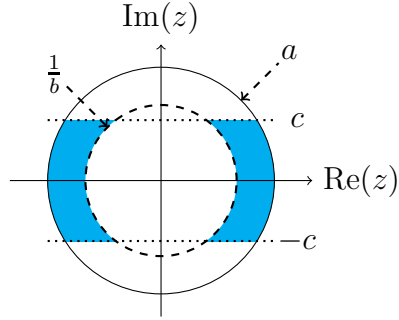

Since Condition \eqref{cond4} holds, the origin is not surrounded by $\Omega_{CG}$, and it follows from Theorem~\ref{thm:exp_decay} that there is always a polynomial  (of some degree) with value 1 at the origin that has a maximum magnitude strictly less than 1 on the closure of this set and hence
GMRES converges with an asymptotic rate given by $\exp(-g(0))<1$, where $g$ is the Green's function of this set with a pole at $\infty$ \cite{choi2015roots,crouzeix2019spectral}.
\end{proof}

\begin{remark} If  \eqref{cond4} in Lemma \ref{lem:fov_main} does not hold, the iterative solver may still converge but we cannot prove convergence using our technique of proof. Specifically, it is immediate to see that $\Omega_{CG}$ is connected and due to the maximum modulus principle, we can only obtain $\min_{p \in \mathbf{P}_j, p(0) = 1} \max_{z \in \Omega_{CG}}| p(z) | \chen{=} 1$, which does not indicate convergence; see Figure~\ref{fig:acge1} for a graphical illustration.
\end{remark}
\begin{figure}[htbp]
    \centering
  \begin{tikzpicture}[scale=1.2]
   
    \begin{scope}
      \clip (-2, -0.8) rectangle (2, 0.8);
      \fill[cyan, opacity=0.3] (0, 0) circle (1.5);
      \fill[white] (0, 0) circle (0.5);
    \end{scope}
    
    \draw[->] (-2, 0) -- (2, 0) node[right] {$\text{Re}(z)$};
    \draw[->] (0, -1.8) -- (0, 1.8) node[above] {$\text{Im}(z)$};
  
    \draw[black] (0, 0) circle (1.5);
  
    \draw[black, thick, dashed] (0, 0) circle (0.5);
  
    \draw[black, thick, dotted] (-1.5, 0.8) -- (1.5, 0.8);
    \draw[black, thick, dotted] (-1.5, -0.8) -- (1.5, -0.8);
    
    \node at (1.8, 0.8) {$c$};
    \node at (1.8, -0.8) {$-c$};
    
    \draw[<-, dashed, thick] (45:1.5) -- (45:2.2);
    \node at (45:2.4) {$a$};
    
    \draw[<-, dashed, thick] (135:0.5) -- (135:1.8);
    \node at (135:2.0) {$\frac{1}{b}$};
    
  \end{tikzpicture}
  \caption{The shaded region is $\Omega_{D}$ when $bc \ge 1$ (i.e., when  \eqref{cond4} in Lemma \ref{lem:fov_main} is violated)}
  \label{fig:acge1}
\end{figure}
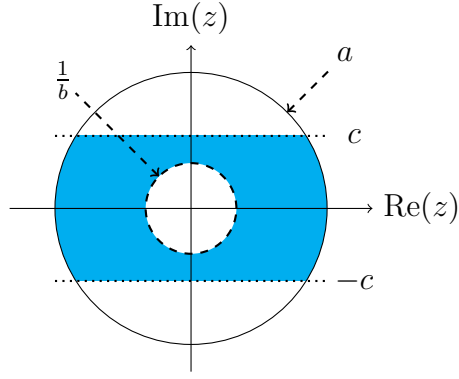
\subsection{Scope and Limitations of the Analysis}

\chen{The use of quadratic forms to establish the notion of field of values-based bounds goes back to early work that studied residual-minimizing iterative methods \cite[Theorem 3.3]{eisenstat1983variational}. We provide here the widely used definition of FOV equivalence.}
\begin{definition}\label{def:fov_eq}
  Given two nonsingular matrices $M,A\in \mathbb{R}^{n\times n}$, \chen{and a symmetric positive definite matrix $H \in {\mathbb R}^{n \times n}$,}  $M$ is {\em $H$-field-of-values equivalent to $A$} if there exist positive constants $\alpha,\beta$ independent of $n$ such that 
  \begin{equation}
  \alpha \leq \frac{(MA \xx,\xx)_{H}}{(\xx,\xx)_{H}}, \ \frac{\| MA \xx\|_{H}}{\| \xx\|_{H}}\leq \beta, \chen{\qquad \forall \xx\in\mathbb{R}^n\setminus\{\mathbf{0}\}.}
  \end{equation}
\end{definition}

If $M$ is $H$-field-of-values equivalent to $A$, the FOV of $MA$ is bounded by a well-defined region:
\begin{equation*}
  W_H(MA) \subseteq \Omega_{\text{FOV}} := \{ z:  \alpha \leq \text{Re} (z),\, |z| \leq \beta \}.
\end{equation*}
For a geometric illustration of $\Omega_{\text{FOV}}$, see Figure \ref{fig:fov_ex_0}.

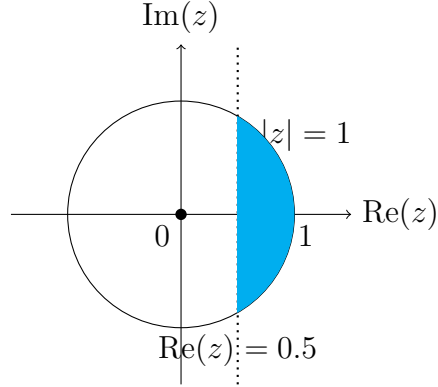
\begin{figure}[h]
  \centering
  \begin{tikzpicture}[scale=1.5]
    \def\alpha{0.5}
    \def\beta{1}
    
    \draw[->] (-1.5,0) -- (1.5,0) node[right] {$\text{Re}(z)$};
    \draw[->] (0,-1.5) -- (0,1.5) node[above] {$\text{Im}(z)$};
    
    \draw (0,0) circle (\beta);
    \node at (1.1*\beta,0.7*\beta) {$|z| = \beta$};
    
    \draw[black, thick, dotted] (\alpha,-1.5) -- (\alpha,1.5);
    \node[black] at (\alpha,-1.2) {$\text{Re}(z) = \alpha$};
    
    \begin{scope}
        \clip (\alpha,-1.5) rectangle (1.5,1.5);
        \clip (0,0) circle (\beta);
        \fill[cyan, opacity=0.3] (\alpha,-2.5) rectangle (1.5,1.5);
    \end{scope}
    
    \fill (0,0) circle (0.05) node[below left] {$0$};
    
    \node[anchor=north] at (\beta + 0.1,0) {$\beta$};
    
  \end{tikzpicture}
  \caption{The shaded region is $\Omega_{\text{FOV}}$ with $\alpha = 0.5$ and $\beta = 1$}
  \label{fig:fov_ex_0}
  \end{figure}

The analysis in \cite{loghin2004analysis} and elsewhere (see, e.g., \cite{klawonn1999block}) pertaining to the case that 0 is not part of the field of values is based on obtaining convergence independent of the matrix dimensions (or mesh size when discretizations of partial differential equations are concerned) by scaling the preconditioner or the inner product. In that case,  Definition \ref{def:fov_eq} is a convergence criterion and it allows for making $\alpha$ and $\beta$ arbitrary (positive) and independent of the matrix dimensions.  

In contrast, in our case,   \eqref{cond4} requires $bc$ to be small. While scaling reduces one of $b$ or $c$, it increases the other. Therefore, a simple scaling strategy does not work in the case we are considering, which reveals a limitation of our analysis. We note that condition \eqref{cond2} is rather standard by norm equivalence considerations (see, for example, \cite[Lemma \chen{2.3}]{loghin2004analysis}). It is condition \eqref{cond3} that seems to present the difficulty, because it requires the skew-symmetric part of the operator to be smaller than the radius of the inner disk; see Figure \ref{fig:aclt1}. Therefore, practically speaking, our analysis is limited to cases where the preconditioned matrix is only mildly nonsymmetric.

  However, we note that this lemma may be improved to allow for looser conditions by using a more sophisticated analysis. 
  \begin{example}
  This is a modified example from \cite{embree2022descriptive}:
  \begin{equation*}
      A = A_{-1}\oplus A_{+1},
  \end{equation*}
where $A_{-1}\in\mathbb{R}^{n\times n}$ and $A_{+1}\in\mathbb{R}^{n\times n}$ are given by $$A_{-1} = \begin{bmatrix}
  -1 & 1/4 & & & \\
  & -1 & 1/4 & & \\
  & & \ddots & \ddots & \\
  & & & -1 & 1/4 \\
  & & & & -1
\end{bmatrix}$$
and $$A_{+1} = \begin{bmatrix}
  2 & 1.2 & & & \\
  & 2 & 1.2 & & \\
  & & \ddots & \ddots & \\
  & & & 2 & 1.2 \\
  & & & & 2
\end{bmatrix}.$$
The field of values of $A$  \chen{is contained in the convex hull} of two disks centered at $-1$ with radius $1/4$ and at $2$ with radius $1.2$. \chen{Note that while the union of these two disks contains the field of values of $A$ independent of the dimension, the field of values itself does depend (mildly) on $n$.} The inverse $A^{-1}$ is available analytically, and it can be shown that 
$\| A^{-1}\|_2^{-1} \to \frac34$ as $n \to \infty$; see, for 
 example, \cite{hansen1991analysis} for useful relevant results for Toeplitz matrices. For a finite value of $n$, the \chen{reciprocal of the norm of $A^{-1}$} needs to be computed numerically, and we have experimentally observed that it is bounded between $0.74$ and $0.76$ for relatively modest values of $n$.

We provide a graphical illustration in Figure \ref{fig:example}. Here $c = 1.2$ and $b \geq \frac{1}{0.76}$. The condition \eqref{cond4} is violated, but GMRES would still converge for a linear system with the matrix $A$  because $\Omega_{CG}$ does not surround/include the origin. A more careful analysis that tracks the boundary of the FOV (see, e.g., \cite{loisel2018path}) might result in conditions that are easier to satisfy.

\chen{ In \cite{GreenbaumWellen2024} there is a detailed discussion and a number of examples, including ones of block diagonal matrices similar to the one in Figure~4, where the field of values is divided into two pieces by removing a disk.}

\begin{figure}[ht]
  \centering
  \includegraphics[width=0.5\textwidth]{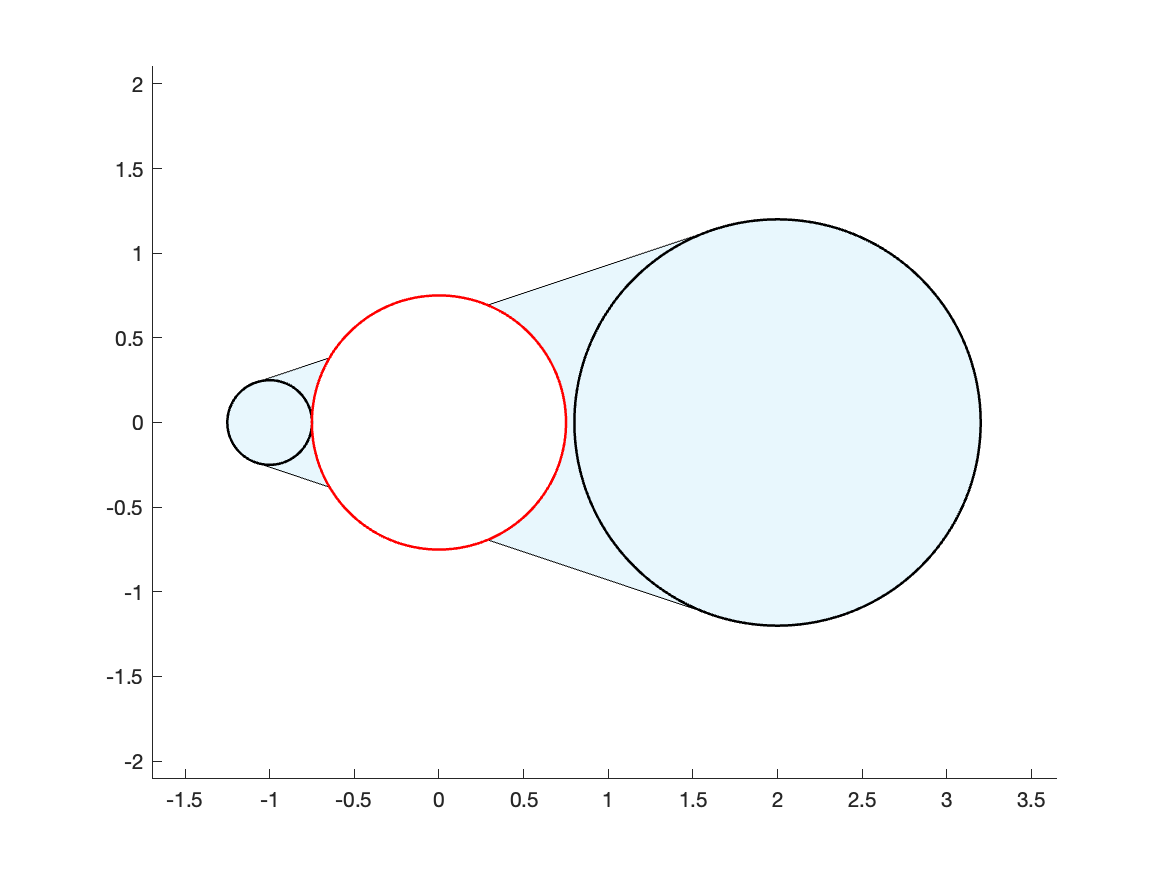}
  \caption{The shaded region is $\Omega_{CG}$ for $A$.}
  \label{fig:example}
\end{figure}
\end{example}

While the limitation we have noted is considerable, our analysis substantially broadens the scope  of preconditioners for which FOV analysis can be carried out. In particular, in terms of the quantities of Definition 
\ref{def:fov_eq}, our analysis makes it possible to consider \chen{the scenario that there exists some nonzero $n$-length real vector $\xx$ for which}
\begin{equation} 
\frac{(MA \xx,\xx)_{H}}{(\xx,\xx)_{H}}\leq 0.
\label{eq:MAxx}
\end{equation}

In the upcoming sections, we present specific examples related to discretized fluid flow problems that demonstrate the advantages and the limitations of our analysis.

\section{Preconditioned Saddle-Point Matrices with Zero in the Field of Values}
\label{sec:fov_sp}

Using the results of Lemma \ref{lem:fov_main}, we now apply our theory to the  important case of a nonsymmetric saddle-point system.

Consider 
\begin{equation}
\label{eq:K}
  K = \begin{bmatrix}
    F & B^T \\
    B & 0
  \end{bmatrix},
\end{equation}
where $F \in \mathbb{R}^{n \times n}$ is nonsingular and $B \in \mathbb{R}^{m \times n}$ is  full row rank. We assume that $F$ is nonsymmetric and positive real (or positive definite), namely, that $\uu^T F \uu> 0$ for all \chen{$\uu \in \mathbb{R}^n\setminus\{\mathbf{0}\}$.}

Let
\begin{equation}
\label{eq:S}
S = B F^{-1} B^T
\end{equation}
be the Schur complement, and define
  \begin{equation} \label{eq:H} H = \begin{bmatrix}
  H_1 & 0 \\
  0 & H_2
\end{bmatrix},
\end{equation}
where $H_1 \in {\mathbb R^{n \times n}}$
and $H_2 \in {\mathbb R^{m \times m}}$ are symmetric positive definite.

To be able to perform our analysis, we need to make some specific assumptions on $H_1$ and $H_2$. We note that these assumptions amount to sufficient conditions, and in practice one may relax them. 

\begin{definition}
    We set $H_1$ as the symmetric part of $F$, and define $N$ as its skew-symmetric part: 
 \begin{equation} \label{eq:sym_skewsym} F = H_1 +N,\qquad H_1 = \frac{F+F^T}{2}, \qquad N=\frac{F-F^T}{2}.\end{equation}
Note that $H_1$ is symmetric positive definite by our  assumptions on $F$.
\label{defn:H1N}
\end{definition}

\begin{assumption} 
\chen{
\begin{equation} \label{req:1}
   \| N\|_{H_1, H_1 ^{-1}} \leq \ttt,
\end{equation}
where  $\ttt$ is a constant independent of the matrix dimensions.
\label{asmp1}}
\end{assumption}
\begin{lemma}  \label{lem:F}
A bound on the weighted norm of $F$ is given by
  \begin{align*}
    \| F\|_{H_1, H_1^{-1}}\leq (1 + \ttt).
  \end{align*}
\end{lemma}
\begin{proof}
    This is immediate from equation \eqref{req:1}.
\end{proof}
\begin{lemma}\label{lem:F_inv} The inverse of $F$ satisfies 
  \[ 
  \| F^{-1} \|_{H_1^{-1},H_1}  \leq 1.
  \]
\end{lemma}
\begin{proof} The result can be readily deduced by using  standard properties of norms; see  \cite[Lemma \chen{2.1}] {loghin2004analysis}. 
We have
   \begin{align*}
    \| F^{-1} \|_{H_1^{-1},H_1} ^{-1} &= \min_{\vv\in \mathbb{R}^n \backslash \{\mathbf{0}\}} \max_{\ww \in \mathbb{R}^n \backslash \{\mathbf{0}\}} \frac{\ww^T F \vv}{\| \vv\|_{H_1} \| \ww\|_{H_1}} \\
    &\geq \min_{\vv\in \mathbb{R}^n \backslash \{\mathbf{0}\}} \frac{\vv^T F \vv}{\| \vv\|_{H_1}^2} \chen{=} 1,
  \end{align*}
\end{proof}

In the problems that we consider, we will assume boundedness of $B$
and a standard inf-sup condition, both of which in fact impose a condition on the choice of $H_2$.
\begin{assumption}
\label{assmp:H2}
\begin{align} \label{eq:B}
\qquad  \| B \|_{H_1, H_2^{-1}} \leq C_1, \qquad \min_\xx \frac{\| B^T \xx \|_{H_1^{-1}}}{\| \xx\|_{H_2}}\geq C_2,
    \end{align}
    where $C_1$ and $C_2$ are independent of $\ttt$ and the dimensions of \chen{$B$}.
\end{assumption}
\begin{lemma}\label{lem:B} If Assumption \ref{assmp:H2} holds, 
    then
  \[ 
  \| S^{-1} \|_{H_2^{-1}, H_2} \leq \schur.
  \] 
\end{lemma}
\begin{proof} Using \eqref{eq:B} and following similar steps to the analysis of \cite{loghin2004analysis}, we have
  \begin{align*}
    \| S^{-1} \|_{H_2^{-1}, H_2}^{-1}  &= \min_{\vv\in \mathbb{R}^m \backslash \{\mathbf{0}\}} \max_{\ww \in \mathbb{R}^m \backslash \{\mathbf{0}\}} \frac{\ww^T B F^{-1} B^T \vv}{\| \vv\|_{H_2} \| \ww\|_{H_2}} \\
      &\geq \min_{\vv\in \mathbb{R}^m \backslash \{\mathbf{0}\}} \frac{\vv^T B F^{-1} B^T \vv}{\| \vv\|_{H_2}^2} \\
      &\geq  \min_{\vv\in \mathbb{R}^m \backslash \{\mathbf{0}\}} \frac{\vv^T B F^{-1} B^T \vv}{\vv^T B H_1 ^{-1} B^T \vv} \min_{\vv\in \mathbb{R}^m \backslash \{\mathbf{0}\}} \frac{\| B^T \vv\|_{H_1^{-1}}^2}{ \| \vv\|_{H_2}^2}\\
      &\geq C_2 ^2 \min_{\vv\in \mathbb{R}^m \backslash \{\mathbf{0}\}} \frac{\vv^T B F^{-1} B^T \vv}{\vv^T B H_1 ^{-1} B^T \vv} .
  \end{align*}
Using  \cite[Lemma \chen{3.4}]{loghin2004analysis} and Lemma \ref{lem:F}, we have
\begin{align*}
  \min_{\vv\in \mathbb{R}^m \backslash \{\mathbf{0}\}} \frac{\vv^T B F^{-1} B^T \vv}{\vv^T B H_1 ^{-1} B^T \vv} &\geq \min_{\yy\in \mathbb{R}^n \backslash \{\mathbf{0}\}} \frac{\yy^T (I + H_1^{-1/2}N H_1^{-1/2})^{-1} \yy}{\yy^T \yy}\\
  &= \min_k \text{Re} \left( \frac{1}{\lambda_k(I+H_1^{-1/2}N H_1^{-1/2})} \right)\\
  &= \min_k \text{Re} \left( \frac{1}
  {1+\lambda_k(H_1^{-1/2}N H_1^{-1/2})} \right)\\
  &=  \frac{1}{\max_{k} {\left|\lambda_k(I + H_1^{-1/2}N H_1^{-1/2})\right|}^2}\\
  & \geq \frac{1}{\| H_1^{-1/2} F H_1^{-1/2} \|_2^2}\\
  &\geq \frac{1}{(1+\ttt)^2},
\end{align*}
\chen{where $\lambda_k(\cdot)$ denotes the $k$th eigenvalue of the input matrix.}
\end{proof}

Finally, we establish notation that will become handy in the following subsections.
\begin{definition}
\label{defn:lesssim}
\chennew{Let $\mathcal{T}$ denote a set of matrices. For a given scalar $\tau>0$, we say that $\|T\| \lesssim \tau$ for $T \in \mathcal{T}$ if there exists some constant \chennewnew{$C>0$ independent of $T$} such that $\|T\| \le C \tau$ for all $T \in \mathcal{T}$. }
\end{definition}

\chennew{\begin{definition}
Consider saddle-point matrices $K$ defined in \eqref{eq:K}, where $F$ is nonsymmetric and positive real, $B$ has full rank, and $H_1, H_2$ and $N$ are as defined in Eq. \eqref{eq:H} and in Definition~\ref{defn:H1N}. We  define $\mathcal{K}_{\eta,C_1,C_2}$ as the set of all such saddle-point matrices that satisfy Assumptions  \ref{asmp1} and \ref{assmp:H2} for given values of $\eta, C_1,$ and $C_2$,  
\end{definition}}

\chennew{In the convergence analysis that follows, our results are  for matrices $K$ in the  $\mathcal{K}_{\eta,C_1,C_2}$ set.}

\subsection{Block-Triangular Preconditioners} \label{sec:btpc}
Let us consider two block-triangular preconditioners:
\begin{enumerate}
\item[(i)] upper block-triangular preconditioners of the form
    \begin{align}
    \label{eq:MU}
  M_U &= \begin{bmatrix}
    F & B^T \\
    0 & H_2
  \end{bmatrix},
\end{align}
with left preconditioning under the $H$-norm;
\item[(ii)] lower block-triangular preconditioners of the form
    \begin{align}
    \label{eq:ML}
  M_L &= \begin{bmatrix}
    F & 0 \\
    B & H_2
  \end{bmatrix},
\end{align}
with right preconditioning under the $H^{-1}$-norm. 
\end{enumerate}
It is well known that there are some differences in the use of left and right preconditioners. For example, in flexible GMRES it is necessary to use right preconditioning. 
The correct norm considered in GMRES for finite element discretizations is typically  $\| \cdot\|_{H^{-1}}$ \cite{arioli2001stopping}. 

Consider first the left preconditioner $M_U$.
The preconditioned matrix is  given by 
\begin{align*}
  M_U^{-1} K = \begin{bmatrix}
    I-F^{-1}B^T H_2^{-1}B & F^{-1}B^T \\
    H_2^{-1} B & 0
  \end{bmatrix} ,
\end{align*}
and its inverse, which is required in order to be able to use Lemma \ref{lem:fov_main}, is given by \begin{align*}
  (M_U^{-1} K)^{-1} = \begin{bmatrix}
I-F^{-1}B^T S^{-1} B & F^{-1}B^T S^{-1} H_2 \\ 
S^{-1} B & I - S^{-1}H_2
  \end{bmatrix} ,
\end{align*}
\chennew{where we recall from \eqref{eq:S} that $S=B F^{-1} B^T$.}
We now need to establish \eqref{cond4} and \eqref{cond1}--\eqref{cond3} in Lemma \ref{lem:fov_main}.

\begin{lemma}[Proof of condition \eqref{cond1} for $M_U$] 
\label{lem:1a}
The $H$-norm of the inverse of the preconditioned matrix associated with the preconditioner $M_U$ satisfies
  \[
  \| (M_U ^{-1} K)^{-1} \|_H \lesssim 1.
  \]
 \end{lemma}
\begin{proof}
The proof follows from putting together the bounds of Lemmas \ref{lem:F}, \ref{lem:F_inv}, and \ref{lem:B}. 
\begin{align*}
  \| (M_U ^{-1} K)^{-1} \|_H &=\left\| \begin{bmatrix}
H_1^{1/2}(I-F^{-1}B^T S^{-1} B) H_1^{-1/2} & H_1^{1/2}(F^{-1}B^T S^{-1} H_2)H_2^{-1/2} \\
H_2^{1/2}(S^{-1} B)H_1^{-1/2} & H_2^{1/2}(I - S^{-1}H_2)H_2^{-1/2}
\end{bmatrix} \right\|_2 \\
&\leq \| I - F^{-1}B^T S^{-1} B\|_{H_1} + \| F^{-1}B^T S^{-1} H_2 \|_{H_2, H_1} + \| S^{-1} B\|_{H_1, H_2} + \| I - S^{-1}H_2\|_{H_2} \\
&\leq (1+2C_1^2 \schur ) + (\schur C_1) + 1 + \schur \\
& \lesssim 1,
\end{align*}
\chen{where in the transition from the second to the third inequality we have used
$$\| I - F^{-1}B^T S^{-1} B\|_{H_1} \leq 1 +\|F^{-1}\|_{{H_1^{-1},H_1}} \|B^T\|_{H_2,H_1^{-1}} \|S^{-1}\|_{H_2^{-1}, H_2} \| B\|_{H_1, H_2^{-1}}, $$ and similarly for the other terms.}
\end{proof}

\begin{lemma}[Proof of condition \eqref{cond2} for $M_U$] The $H$-norm of the preconditioned matrix associated with the preconditioner $M_U$ satisfies
\label{lem:1b}
 \begin{align*}
  \| M_U^{-1}K \|_{H} \lesssim 1.
 \end{align*}
 \end{lemma}

\begin{proof} Similarly to the proof of Lemma~\ref{lem:1a},
  \begin{align*}
    \| M_U^{-1}K \|_{H} &= \left\| \begin{bmatrix}
   H_1^{1/2}(I - F^{-1}B^T H_2^{-1}B)H_1^{-1/2} & H_1^{1/2}F^{-1}B^T H_2^{-1/2} \\
    H_2^{1/2}H_2^{-1}B H_1^{-1/2} & 0
      \end{bmatrix} \right\|_2 \\
      &\leq \| I - F^{-1}B^T H_2^{-1}B\|_{H_1} + \| F^{-1}B^T \|_{H_2, H_1} + \| B\|_{H_1, H_2^{-1}} \\
      &\leq (1 + \schur C_1^2 ) + 2C_1\\
      &\lesssim 1.
  \end{align*}
\end{proof}

\begin{lemma}[Proof of condition \eqref{cond3} for $M_U$]
\label{lem:1c}
When $\ttt < \frac{1}{2}$, we have
  \begin{align}
    \left\|H(M_U^{-1}K) - (M_U^{-1}K)^T H\right\|_{H, H^{-1}} \lesssim \ttt.
    \label{eq:2c}
  \end{align}
\end{lemma}
\begin{proof}
  We have \begin{equation}\label{eq:inv_ineq} 
    \begin{aligned}
    \left\|H(M_U^{-1}K) - (M_U^{-1}K)^T H\right\|_{H, H^{-1}} &= 
    \left\| \begin{bmatrix}
      B_{11} & B_{12} \\ 
      -B_{12}^T & 0
         \end{bmatrix} \right\|_2 \\
         &\leq \| B_{11} \|_2 + 2\| B_{12} \|_2 ,
        \end{aligned}
  \end{equation}
  where
  $$
  B_{11 }= -H_{1}^{1/2}F^{-1}B^T H_2^{-1} B H_1^{-1/2} + H_1^{-1/2}B^T H_2^{-1}B F^{-T}H_1^{1/2}.
  $$
  $$
  B_{12} = H_1^{1/2}F^{-1}B^T H_2^{-1/2} - H_1^{-1/2} B^T H_2^{-1/2}
  $$
  \begin{align*}
    \| B_{12} \|_2 &= \| (H_1 F^{-1} -I )B^T \|_{H_2, H_1^{-1}} \\
    &\leq C_1 \| H_1 F^{-1} -I \|_{H_1^{-1}} \\
    &= C_1 \| H_1^{1/2} (H_1 +  N)^{-1} H_1^{1/2} - I \|_2\\
    &= C_1 \|  (I + H_1^{-1/2}N H_1^{-1/2})^{-1}  - I \|_2.
  \end{align*}
  When $\ttt < \frac{1}{2}$ we have 
  \begin{equation}\label{eq:B_12}
    \| B_{12} \|_2 \leq C_1 \frac{ \| N\|_{H_1,H_1^{-1}} }{1 -  \| N\|_{H_1,H_1^{-1}} } \leq C_1 \ttt / (1-\ttt) \leq 2C_1 \ttt \lesssim \ttt
  \end{equation}
and
  \begin{align*}
    \| B_{11} \| &= \|H_{1}F^{-1}B^T H_2^{-1} B- B^T H_2^{-1}B F^{-T}H_1\|_{H_1,H_1^{-1}} \\
    &= \| (F -  N) F^{-1}B^T H_2^{-1} B- B^T H_2^{-1}B F^{-T}(F^T -  N^T) \|_{H_1,H_1^{-1}} \\
    &=\| -  NF^{-1}B^T H_2^{-1} B +  B^T H_2^{-1} B F^{-T} N^T \|_{H_1,H_1^{-1}} \\
    &\leq  \|NF^{-1}B^T H_2^{-1} B\|_{H_1,H_1^{-1}} + \| B^T H_2^{-1} B F^{-T} N^T \|_{H_1,H_1^{-1}} \\
    &\leq 2C_1^2 \ttt  \\
    &\lesssim \ttt.
  \end{align*}
Substituting  the above inequalities into \eqref{eq:inv_ineq}, we obtain \eqref{eq:2c}, as required.
\end{proof}

\chen{In order to be able to state an upcoming convergence theorem in precise terms, it is useful to rephrase the results of Lemmas \ref{lem:1a}--\ref{lem:1c} using some specific constants. 
\chennew{Suppose 
$K \in \mathcal{K}_{\eta,C_1,C_2}$ for some particular values of $\ttt < \frac12$, $C_1$, and $C_2$. Then Lemmas \ref{lem:1a}--\ref{lem:1c}}
say that there exist constants $C_a^U, C_b^U, C_c^U>0$ \chennewnew{independent of the matrix dimensions} such that
 \begin{equation}
\begin{aligned}
  &\| M_U ^{-1} K \|_H \le C_a^U,\quad  \| (M_U ^{-1} K)^{-1} \|_H \le C_b^U,\\
  &\left\|H(M_U^{-1}K) - (M_U^{-1}K)^T H\right\|_{H, H^{-1}} \le C_c^U\,\ttt.
\end{aligned}
\label{eq:CaCbCc}
\end{equation}}
   
\begin{theorem}  \label{thm:leftpc}
Given \chennew{ a tolerance \chennewnew{$\epsilon >0$} and a saddle-point system with matrix $K \in 
\mathcal{K}_{\eta,C_1,C_2}$ for some particular values of $\ttt < \frac12, C_1$, and $C_2$, } let $H_1$ and $N$ be the symmetric and skew-symmetric parts, respectively, of $F$, as in \eqref{eq:sym_skewsym}.
 Let $H_2$ be a symmetric positive definite matrix. 
Finally, let $H$ be the block-diagonal matrix defined in \eqref{eq:H}.
Then, using the constants defined in \eqref{eq:CaCbCc}, if $C_b^U C_c^U\,\ttt < 1$ (i.e., \eqref{cond4} holds), then for any system that satisfies these requirements with  $\ttt$, $C_1$, $C_2$, $C_a^U$, $C_b^U$,  and $C_c^U$, there exists some integer $m\ge 1$ that depends on their values \chennew{and on $\epsilon$}, such that preconditioned GMRES with $M_U$ as a left preconditioner  will converge (in the $H$-norm) \chennew{to $\epsilon$} in no more than $m$ iterations. 
\end{theorem}
\begin{proof}
Trivially, by Lemmas \ref{lem:1b} and \ref{lem:1c}, \eqref{cond4} holds when $\ttt$ is sufficiently small. Lemmas \ref{lem:1a}--\ref{lem:1c} validate conditions \eqref{cond1}--\eqref{cond3}.
\end{proof}

We now consider the right preconditioner $M_L$ defined in \eqref{eq:ML}.
The analysis is very similar to the left preconditioner case. 
\chen{We first define three inequalities analogous to \eqref{eq:CaCbCc}:
 \begin{equation}
\begin{aligned}
  &\| K M_L^{-1} \|_{H^{-1}} \le C_a^L,\quad  \| (K M_L ^{-1} )^{-1} \|_{H^{-1}} \le C_b^L,\\
  &\left\|H^{-1} (K M_L^{-1}) - (K M_L^{-1})^T H^{-1}\right\|_{H^{-1},  H} \le C_c^L\,\ttt.
\end{aligned}
\label{eq:CaCbCcL}
\end{equation}
 We now present a theorem analogous to Theorem \ref{thm:leftpc}.}
\begin{theorem}
  \label{thm:rightpc}
Given \chennew{ a tolerance \chennewnew{$\epsilon>0$} and a saddle-point system with matrix $K \in 
\mathcal{K}_{\eta,C_1,C_2}$ for some particular values of $\ttt <\frac12, C_1$, and $C_2$, } let $H_1$ and $N$ be the symmetric and skew-symmetric parts, respectively, of $F$, as in \eqref{eq:sym_skewsym}.
 Let $H_2$ be a symmetric positive definite matrix.
Let $H$ be the block-diagonal matrix defined in \eqref{eq:H}.
\chen{Finally, assume we have the  constants and inequalities defined in \eqref{eq:CaCbCcL}. Then,  if $C_b^L C_c^L\,\ttt < 1$ (i.e., \eqref{cond4} holds), then for any system that satisfies these requirements with  $\ttt$, $C_1$, $C_2$, $C_a^L$, $C_b^L$, and $C_c^L$, there exists some integer $m\ge 1$ that depends on their values \chennew{and on $\epsilon$}, such that preconditioned GMRES with $M_L$ as a right preconditioner  will converge (in the $H^{-1}$-norm) \chennew{to $\epsilon$} in no more than $m$ iterations. }
\end{theorem}

\begin{remark}
  In practice, $H$ can be replaced with another symmetric positive definite matrix $\tilde{H}$ and results will still hold if $H$ and $\tilde{H}$ are spectrally equivalent: GMRES convergence with $H$-norm can induce GMRES convergence with $\tilde{H}$-norm. This is because 
  $$\|p(A) \|_{H} =  \| H^{1/2} (\tilde{H}^{-1/2} \tilde{H}^{1/2}) p(A) (\tilde{H}^{-1/2} \tilde{H}^{1/2}) H^{-1/2} \|_2 \leq \kappa_2( H^{1/2} \tilde{H}^{-1/2} ) \| p(A)\|_{\tilde{H}}.$$
\end{remark}

\subsection{A Block-Diagonal Preconditioner} 

The case of a block diagonal preconditioner of the form 
\begin{align}
\label{eq:MD}
  M_D &= \begin{bmatrix}
    F & 0 \\
    0 & H_2
  \end{bmatrix}
\end{align}
is interesting in the context of this work, because  contrary to block-triangular preconditioners, where one might select either an upper block-triangular preconditioner or a lower block-triangular preconditioner along with right or left preconditioning to avoid a situation of having zero in the field of values, here it is immediate that the field of values contains zero regardless of any such choices made. There is no practical difference between left and right preconditioning here, and we proceed with left preconditioning below.
The preconditioned matrix is
\begin{align*}
  M_D^{-1} K = \begin{bmatrix}
    I & F^{-1}B^T \\
    H_2^{-1} B & 0
  \end{bmatrix} ,
\end{align*}
and its inverse is
\begin{align*}
  (M_D^{-1} K)^{-1} = \begin{bmatrix}
    I - F^{-1}B^T S^{-1} B & F^{-1}B^T S^{-1}H_2 \\
    S^{-1} B & -S^{-1} H_2
  \end{bmatrix} .
\end{align*}
The analysis is essentially identical to that 
of Section~\ref{sec:btpc}.

\begin{lemma}[Proof of condition \eqref{cond1} for $M_D$]
The $H$-norm of the inverse of the preconditioned matrix associated with the preconditioner $M_D$ satisfies
 \label{lem:MD_property1}
  \[
  \| (M_D^{-1} K)^{-1} \|_H \lesssim 1.
  \]
\end{lemma}
\begin{proof}
  The proof follows similar steps as for \( M_U \) in Lemma \ref{lem:1a}. We need to bound the norm of each block in the inverse, and we apply the bounds obtained in Lemmas \ref{lem:F}, \ref{lem:F_inv}, and \ref{lem:B}:
  \begin{align*}
    \| I - F^{-1}B^T S^{-1} B \|_{H_1} &\leq 1 + C_1^2 \schur, \\
    \| F^{-1}B^T S^{-1}H_2 \|_{H_2, H_1} &\leq \schur C_1, \\
    \| S^{-1} B \|_{H_1, H_2} &\leq C_1, \\
    \| S^{-1} H_2 \|_{H_2} &\leq \schur.
  \end{align*}
  Combining these, we get the bound for the entire matrix.
\end{proof}

\begin{lemma}[Proof of condition \eqref{cond2} for $M_D$]
The $H$-norm of the preconditioned matrix associated with the preconditioner $M_D$ satisfies 
  \label{lem:MD_property2}
  \[
  \| M_D^{-1} K \|_H \lesssim 1.
  \]
\end{lemma}
\begin{proof}
  Similar to the analysis for \( M_U \) in Lemma \ref{lem:1b}, we bound the norm of each block in the preconditioned matrix:
  \begin{align*}
    \| I \|_{H_1} &= 1, \\
    \| F^{-1}B^T \|_{H_2, H_1} &\leq C_1, \\
    \| H_2^{-1} B \|_{H_1, H_2} &\leq C_1.
  \end{align*}
  Thus, the norm of the entire matrix is bounded by the sum of these norms.
\end{proof}

\begin{lemma}[Proof of condition \eqref{cond3} for $M_D$]
  \label{lem:M2_property3} When $\ttt < \frac{1}{2}$, we have
  \[
  \left\| H(\chen{M_D}^{-1} K) - (\chen{M_D}^{-1} K)^T H \right\|_{H, H^{-1}} \lesssim \ttt.
  \]
\end{lemma}

\begin{proof}
  Note that
  \begin{equation*}
    \begin{aligned}
    \left\|H(M_D^{-1}K) - (M_D^{-1}K)^T H\right\|_{H, H^{-1}} &= 
    \left\| \begin{bmatrix}
      0& B_{12} \\ 
      -B_{12}^T & 0
         \end{bmatrix} \right\|_2 \\
         &\leq 2\| B_{12} \|_2 ,
        \end{aligned}
  \end{equation*}
  where 
  $$\| B_{12} \|_2 = \| (H_1 F^{-1} -I )B^T \|_{H_2, H_1^{-1}}. $$
  By \eqref{eq:B_12}, we complete the proof.
\end{proof}
\chen{
Here we need two sets of inequalities, analogous both to \eqref{eq:CaCbCc} and \eqref{eq:CaCbCcL}:
\begin{equation}
\begin{aligned}
  &\| M_D ^{-1} K \|_H \le C_a^{D_1},\quad  \| (M_D ^{-1} K)^{-1} \|_H \le C_b^{D_1},\\
  &\left\|H(M_D^{-1}K) - (M_D^{-1}K)^T H\right\|_{H, H^{-1}} \le C_c^{D_1}\,\ttt
\end{aligned}
\label{eq:CaCbCcD1}
\end{equation}
and
\begin{equation}
\begin{aligned}
  &\| K M_D^{-1} \|_{H^{-1}} \le C_a^{D_2},\quad  \| (K M_D ^{-1} )^{-1} \|_{H^{-1}} \le C_b^{D_2},\\
  &\left\|H^{-1} (K M_D^{-1}) - (K M_D^{-1})^T H^{-1}\right\|_{H^{-1},  H} \le C_c^{D_2}\,\ttt.
\end{aligned}
\label{eq:CaCbCcD2}
\end{equation}}
The convergence theorem is then given as follows.
\begin{theorem}
  \label{thm:diagpc}
 Given \chennew{ a tolerance \chennewnew{$\epsilon >0$} and a saddle-point system with matrix $K \in 
\mathcal{K}_{\eta,C_1,C_2}$ for some particular values of $\ttt <\frac12, C_1$, and $C_2$, } let $H_1$ and $N$ be the symmetric and skew-symmetric parts, respectively, of $F$, as in \eqref{eq:sym_skewsym}.
 Let $H_2$ be a symmetric positive definite matrix.
Let $H$ be the block-diagonal matrix defined in \eqref{eq:H}.
\chen{Finally, assume we have the  constants and inequalities defined in \eqref{eq:CaCbCcD1} and \eqref{eq:CaCbCcD2}. Then,  if $C_b^{D_1} C_c^{D_1}\,\ttt < 1$ $C_b^{D_2} C_c^{D_2}\,\ttt < 1$  (i.e., \eqref{cond4} holds for both), then for any system that satisfies these requirements with  $\ttt$, $C_1$, $C_2$, $C_a^{D_1}$, $C_a^{D_2}$, $C_b^{D_1}$, $C_b^{D_2}$, $C_c^{D_1}$, and $C_c^{D_2}$ there exists some integer $m\ge 1$ that depends on their values \chennew{and on $\epsilon$} such that preconditioned GMRES with $M_D$ as a right preconditioner  will converge in the $H^{-1}$-norm and preconditioned GMRES with $M_D$ as a left preconditioner  will converge in the $H$-norm  \chennew{to $\epsilon$} in no more than $m$ iterations. }
\end{theorem}

\subsection{Inexact Preconditioning}
To make the iterations practical, one needs to consider computationally inexpensive ways of approximately inverting the  preconditioners that we have discussed so far, and using those approximate linear operators as the actual preconditioners. Under mild conditions, our analysis seems to carry over to such situations. We illustrate this for a block upper-triangular preconditioner that approximates  the leading block. Consider
$$
\tilde{M}_U = \begin{bmatrix}
  P_1 & B^T \\ 
  0 & H_2
\end{bmatrix},
$$
where the action of (implicitly) inverting $P_1$ is computationally practical.
Note that
$$
\tilde{M}_U ^{-1} K = (\tilde{M}_U^{-1} M_U) M_U^{-1}K
$$
and 
$$
\tilde{M}_U^{-1} M_U = \begin{bmatrix}
 P_1^{-1} F & 0\\
 0 & I
\end{bmatrix}.
$$
\begin{assumption}
    We assume $\|P_1^{-1} F - I\|_{H_1} \leq C_3 \ttt$ and $\| F^{-1} P_1\|_{H_1} \leq C_4$.
    \label{assum:P1}
\end{assumption}
\noindent Based on Assumption~\ref{assum:P1}, we have
$$
\| \tilde{M}_U^{-1} M_U\|_{H} \leq (1 + C_3 \ttt) + 1 \lesssim 1
$$
and 
$$
\| (\tilde{M}_U^{-1} M_U)^{-1}\|_{H} \leq \| F^{-1} P_1\|_{H_1} +  1 \lesssim 1.
$$
We now examine the sufficient conditions.
For condition \eqref{cond1}, we have 
$$
\| \tilde{M}_U^{-1} K\|_{H} \leq \| \tilde{M}_U^{-1} M_U\|_{H} \| M_U^{-1} K\|_{H} \leq (1 + C_3 \ttt) \| M_U^{-1} K\|_H\lesssim 1.
$$
For condition \eqref{cond2}, we have 
$$
\| (\tilde{M}_U^{-1} K)^{-1}\|_{H} \leq \| (\tilde{M}_U^{-1} M_U)^{-1}\|_{H} \| (M_U^{-1} K)^{-1}\|_{H} \lesssim 1.
$$
For condition \eqref{cond3}, we have 
\begin{align*}
\| H(\tilde{M}_U^{-1} K) - (\tilde{M}_U^{-1} K)^T H\|_{H, H^{-1}} &\leq \| H(M_U^{-1} K) - (M_U^{-1} K)^T H\|_{H, H^{-1}}  \\&+ \| H(\tilde{M}_U^{-1}M_U - I) M_U^{-1}K - (M_U ^{-1}K)^T (\tilde{M}_U^{-1}M_U - I)^T H\|_{H, H^{-1}} \\
&\lesssim \ttt +2 \| P_1 ^{-1} F - I\|_{H_1} \|M_U^{-1} K \|_{H}\\
&\lesssim \ttt.
\end{align*}
Thus, if $\ttt$ is small enough,  \eqref{cond4} is satisfied and the iterative solver with $\tilde{M}_U$ as a preconditioner will converge in a fixed number of iterations.

\section{Numerical Experiments} 
\label{sec:experiments}
We provide a couple of examples of applications from fluid dynamics to validate our analysis. 
\subsection{Navier-Stokes Equations}
\label{sec:ns}
Let $\Omega \subset \mathbb{R}^2$ be a bounded domain. 
The Navier-Stokes equations with pure Dirichlet boundary conditions are 
\begin{align*}
  -\nu \Delta \uu + (\uu \cdot \nabla) \uu + \nabla p = \mathbf{f} \quad \text{in } \Omega, \\
  \nabla \cdot \uu = 0 \quad \text{in } \Omega, \\
  \uu = {\uu}_d \quad \text{on } \partial \Omega,
\end{align*}
where $\nu$ is a viscosity coefficient,  $\uu$ is the velocity, $p$ is the pressure, and $\uu=\uu_d$ provides the Dirichlet boundary conditions. 
\chen{For an in-depth description of the finite element solution of the Navier-Stokes equations, see, for example, 
  \cite{elman2005finite,girault2012finite}.}
  
Linearizing the equations using the Picard iteration, we obtain 
\begin{align*}
  -\nu \Delta \uu + (\mathbf{b} \cdot \nabla) \uu + \nabla p = \mathbf{f} \quad \text{in } \Omega,\\
  \nabla \cdot \uu = 0 \quad \text{in } \Omega, \\
  \uu = {\uu}_d \quad \text{on } \partial \Omega,
\end{align*}
where $\mathbf{b}$ is the velocity from the previous iteration.

For simplicity, we assume 
$\uu_d={\mathbf 0}.$
Define the Sobolev spaces
\begin{equation*}
  \mathbf{V} = \{\mathbf{v} \in  (H^1(\Omega))^2: \mathbf{v}=0 \text{ on } \partial \Omega\}\, , \quad Q = \{ q\in L^2(\Omega): \int_{\Omega} q = 0 \} .
\end{equation*}
The weak form involves solving the following system:
find $\uu \in \mathbf{V}$ and $p \in Q$ such that
\begin{equation*}
  \begin{aligned}
  a(\uu,\vv) + b(\vv,p) &= f(\vv) \quad \forall \vv \in \mathbf{V}, \\
  b(\uu,q) &= 0 \quad \forall q \in Q,
  \end{aligned}
\end{equation*}
where the bilinear forms are defined as
\begin{equation*}
  \begin{aligned}
  a(\uu,\vv) &= \nu \int_{\Omega} \nabla \uu \cdot \nabla \vv + \int_{\Omega} (\mathbf{b} \cdot \nabla \uu) \cdot \vv, \\
  b(\uu,q) &= -\int_{\Omega} (\nabla \cdot \uu)q,
  \end{aligned}
\end{equation*}
and $f({\mathbf v})=\int_{\Omega} \mathbf{f} \cdot \vv.$

By using conforming finite element spaces ${\mathbf V}_h \subset {\mathbf V}$ and $Q_h \subset Q$, we discretize these equations and obtain the nonsymmetric saddle-point system
\begin{equation}
  \begin{aligned}
  \begin{bmatrix}
    F & B^T \\
    B & 0
  \end{bmatrix}
  \begin{bmatrix}
    \uu \\
    \pp
  \end{bmatrix}
  =
  \begin{bmatrix}
    \mathbf{f} \\
    \mathbf{0}
  \end{bmatrix},
  \end{aligned}
\end{equation}
where $F  = \nu H_1 + N$, $a(\uu_h,\vv_h) = (F \uu_1,\vv_1) =\nu (H_1 \uu_1,  \vv_1 ) + (N \uu_1, \vv_1) $ , $(\nabla \uu_h, \nabla \vv_h) = (H_1 \uu_1, \vv_1)$, and $((\mathbf{b} \cdot \nabla \uu_h), \vv_h) = (N \uu_1, \vv_1)$. So far, this is a standard treatment of these equations; see~\cite{elman2005finite}.

To make our analysis  applicable, we scale the system on the left by 
$\begin{bmatrix}
  \frac{1}{\nu} & 0 \\
  0 & 1
\end{bmatrix}$ and on the right by 
$\begin{bmatrix}
  1 & 0 \\
  0 & \nu
\end{bmatrix}$, respectively, and the system becomes 
$$
\begin{bmatrix} 
  H_1 + \frac{1}{\nu} N & B^T \\
  B & 0
\end{bmatrix}
\begin{bmatrix} 
  \uu \\
 \frac{1}{\nu} \pp
\end{bmatrix}
=
\begin{bmatrix} 
 \frac{1}{\nu} \mathbf{f} \\
  \mathbf{0}
\end{bmatrix}.
$$

Since the problem is given with pure Dirichlet boundary conditions, we have $N^T = -N$, which indicates $H_1$ is the symmetric part of $H_1+\frac{1}{\nu}N$. If $\nu$ is sufficiently large, then conditions \eqref{req:1} and \eqref{eq:B} are satisfied, as we have used conforming elements.

\begin{remark}
\chennew{
The skew-symmetric part is small in norm by Assumption~\ref{asmp1}, which holds if $\nu$ is large, but the saddle-point matrix is  nonnormal and performing convergence analysis for GMRES is challenging. An  analysis based on eigenvalue and eigenvector conditioning may be possible, based on observations and insights such as those in \cite[Corollary 2.2 and Section 4.1]{embree2022descriptive} or \cite{GreenbaumStrakos1994Krylov}, although we are not aware of a comprehensive analysis of this type for the Navier-Stokes equations.  The field-of-values analysis  is an alternative approach, applicable under the assumptions made.  }
\label{rem:smallskew}
\end{remark}

\chen{
Our setting can be extended to mixed boundary conditions. The bilinear form  
can be writen as
$$
a_{\rm mix}(\uu,\vv) =  \nu \int_{\Omega} \nabla \uu \cdot \nabla \vv + \frac{1}{2}\int_{\Gamma_n}(\mathbf{b}\cdot \mathbf{n}) \uu\cdot \vv+ \int_{\Omega} (\mathbf{b} \cdot \nabla \uu) \cdot \vv - \frac{1}{2}\int_{\Gamma_n}(\mathbf{b}\cdot \mathbf{n}) \uu\cdot \vv.
$$
\chennew{Here, $\Gamma_n$ is the boundary with Robin conditions $\uu \cdot \mathbf{n} = g$ and $\mathbf{n}$ is a vector tangent to the boundary. }
Let us denote the leading block in the saddle-point system in this case by $\tilde{F} = \nu \tilde{H}_1 + \tilde{N}$. Then $$a_{\rm mix}(\uu_h,\vv_h)=(\tilde{F} \uu_1,\vv_1)=\nu(\tilde{H}_1 \uu_1, \vv_1) + (\tilde{N}\uu_1,\vv_1),$$ where $$(\tilde{H}_1 \uu_1,\vv_1) =   \int_{\Omega} \nabla \uu_1 \cdot \nabla \vv_1 + \frac{1}{2\nu}\int_{\Gamma_n}(\mathbf{b}\cdot \mathbf{n}) \uu_1\cdot \vv_1$$ and  $$\int_{\Omega} (\mathbf{b} \cdot \nabla \uu_1) \cdot \vv_1 - \frac{1}{2}\int_{\Gamma_n}(\mathbf{b}\cdot \mathbf{n}) \uu_1\cdot \vv_1 =(\tilde{N}\uu_1,\vv_1). 
$$ 
Note that if $\nu$ is large enough, 
$\tilde{H}_1$ is spectrally equivalent to $H_1$, thus the mixed boundary conditions give us similar results to the setting with Dirichlet boundary conditions.}

We numerically solve the regularized lid-driven cavity problem using IFISS \cite{ifiss} to illustrate our results. The domain $\Omega$ is chosen as $[-1,1]^2$. Zero boundary conditions are imposed, except we take $u_x = 1-x^4$ on $\{y=1,-1\leq x\leq 1\}$.

 We set $\nu=1$, because our analysis requires it to be relatively large, and apply the Picard iteration, \chennew{using the IFISS default nonlinear tolerance}. Since $\nu$ is relatively large, the nonlinear iterations converge quickly; we record average iteration counts and examine the performance of the linear solvers. We note that we have observed no significant differences among the linear solver iteration counts throughout the nonlinear iteration. We use the diagonal preconditioner $M_D$ defined in \eqref{eq:MD} and the upper triangular preconditioner $M_U$ defined in \eqref{eq:MU}.  We use left preconditioning for both;  the results for right preconditioning with $M_L$ defined in \eqref{eq:ML} are virtually the same. 
 
 Results for a few mesh sizes can be found in Table \ref{tab:ns}. We observe an excellent level of scalability: the iteration counts are nearly constant for various mesh sizes in all cases. We present our iteration counts in both the $\ell_2$ and $H$ norms, and observe that they are nearly identical. 
 
\begin{table}[htbp]
\centering
\begin{tabular}{|c|cc|cc|}
  \hline
  System Size & \multicolumn{2}{c|}{$\ell_2$-norm} & \multicolumn{2}{c|}{$H$-norm} \\
  \cline{2-5}
   & Diagonal & Upper Triangular & Diagonal & Upper Triangular \\
  \hline
  210 & 21.0 & 11.0 & 21.0 & 11.3 \\
  770 & 22.5 & 12.0 & 23.0 & 12.0 \\
  2,946 & 23.0 & 12.5 & 23.0 & 12.5 \\
  11,522 & 24.0 & 13.0 & 24.0 & 13.0 \\
  \hline
  \end{tabular}
\caption{Average iteration counts for Navier-Stokes. \chennew{For each Picard iteration, the inner linear system solve was terminated when a relative residual tolerance of $\frac{\| r_k\| }{ \| r_0 \|} < 10^{-5}$ was reached,  where the norms used were the ones corresponding to results reported in the table: $\ell_2$-norm on the left and $H$-norm on the right.}}
\label{tab:ns}
\end{table}

For the diagonal preconditioner, we have computed the parameters of \chen{Lemma}~\ref{lem:fov_main} \chen{in the $H$-norm} and have observed that $b \approx 2.25$ and $c\approx 0.016$. For the upper-triangular preconditioner, $b\approx 2.06$ and $c\approx 0.035$. In both cases we have $bc<1$, as required.

\subsection{Stokes-Darcy Equations}
\label{sec:sd}
Consider the  Stokes-Darcy equations on a non-overlapping domain $\Omega = \Omega_s \cup \Omega_d$ with a polygonal interface $\Gamma_I = \partial \Omega_s \cap \partial \Omega_d$: 
\begin{equation*} 
    \begin{aligned}
    -\nabla \cdot(2\nu D(\mathbf{u})-p\mathbf{I}) &= \mathbf{f}^s & \text{ in } \partial \Omega_s,\\
    \nabla \cdot \mathbf{u} &= 0 & \text{ in } \partial \Omega_s,\\
    \mathbf{u} &= \mathbf{g}^{s} & \text{ on } \Gamma_s = \partial\Omega_s \cap\partial\Omega ,  \\
     -k\Delta \phi &= f^d & \text{ in } \Omega_d,\\
    \phi &= g^d & \text{ on } \Gamma_d,\\
    k \nabla \phi \cdot \mathbf{n} &= g^n & \text{ on } \Gamma_n,  \\
    \mathbf{u}\cdot \mathbf{n}_{12} &= -k\nabla \phi \cdot \mathbf{n}_{12} & \text{ on } \Gamma_{I}, \\
    (-2\nu D(\mathbf{u}) \cdot \mathbf{n}_{12} + p \mathbf{n}_{12})\cdot \mathbf{n}_{12} &= \phi & \text{ on } \Gamma_{I}, \\
    \mathbf{u} \cdot \boldsymbol{\tau}_{12} &= -2\nu G(D(\mathbf{u})\mathbf{n}_{12})\cdot \boldsymbol{\tau}_{12} & \text{ on } \Gamma_{I},
    \end{aligned}
\end{equation*}
where $\mathbf{u}$ satisfies the incompressibility condition $\nabla \cdot \mathbf{u} = 0$. Here, $\Omega_S$ and $\Omega_d$ are assumed to be simple domains, e.g., the unit squares in two dimensions, with a polygonal interface. The operator $D$ is defined as $D(\uu)=\frac12 (\nabla \uu+\nabla\uu^T).$ The physical parameters $\nu$ and $k$ denote the viscosity coefficient and hydraulic constant, respectively. The constant $G$ represents an experimentally-determined constant related to the Beavers-Joseph-Saffman interface condition. Finally, $\mathbf{n}_{12}$ and $\chen{\boldsymbol{\tau}_{12}}$ are  unit normal and  tangential vectors; see \cite{chidyagwai2016constraint} for details.

We use the finite element discretization described in \cite{chidyagwai2016constraint,discacciati2009navier, chidyagwai2010numerical}; some details on the Stokes part are similar to Section \ref{sec:ns}. We note that there are several other distinct possibilities here for different discretizations; see, e.g., \cite{greif2023block}.  Full details on the discretization of the entire problem are omitted since this is not the focus of our paper. The discretization yields the following linear system:
\begin{equation}\label{eq:linear-system}
  \mathcal{K}\begin{bmatrix}
    \mathbf{u}_1\\
    \phi_1\\
    p_1
  \end{bmatrix}=
  \begin{bmatrix}
    \nu \aas & I_{12}^T & B^T\\
    -I_{12} & k A_{\Omega_d} & 0\\
    B & 0 & 0
  \end{bmatrix}
  \begin{bmatrix}
    \mathbf{u}_1\\
    \phi_1\\
    p_1
  \end{bmatrix}
  =
  \begin{bmatrix}
    \mathbf{f}\\
    f\\
    0
  \end{bmatrix},
\end{equation}
where $\uu_1$, $\phi_1$ and $p_1$ represent the vectors of coefficients in the finite element basis expansions, with corresponding continuous finite element solutions denoted by $\uu_h$, $p_h$ and $\phi_h$, respectively. 

For simplicity of our analysis, we assume $k = \nu $,
and consider the following scaled matrix:
\begin{equation*}
  \begin{bmatrix}
    \aas & \frac{1}{\nu} I_{12}^T & B^T\\
    -\frac{1}{\nu}I_{12} &  A_{\Omega_d} & 0\\
    B & 0 & 0
  \end{bmatrix}.
\end{equation*}
Then, assuming that $\nu$ is sufficiently large (which corresponds to requiring $\ttt$ to be sufficiently small in our analysis in Section~\ref{sec:fov_sp}; see \eqref{req:1}), let us define $$F = \begin{bmatrix}
  \aas & 0 \\
  0 & A_{\Omega_d}
\end{bmatrix} + \frac{1}{\nu}\begin{bmatrix}
  0 & I_{12}^T \\
  -I_{12} & 0
\end{bmatrix}.$$ It has been shown in the literature \cite{chidyagwai2016constraint,discacciati2009navier} that the inf-sup condition for the matrix $\begin{bmatrix}
B & 0 
\end{bmatrix}$ is satisfied and  that the skew-symmetric operator $\begin{bmatrix}
  0 & I_{12}^T \\
  -I_{12} & 0
\end{bmatrix}$ is bounded. Therefore, the conditions of Lemma \ref{lem:fov_main} are satisfied, and  it follows that an iterative solver preconditioned with the block preconditioners discussed in Section~\ref{sec:fov_sp} will converge independently of the mesh size.

We use the following example from \cite{chidyagwai2016constraint}.
We choose $\Omega_s$ to be $[0,1]^2$ and $\Omega_d$ to be $[0,1]\times [1,2]$. $\Gamma_n$ is $\{x=0,y \in [0,1]\}\cup\{x=1,y \in [0,1]\}$. Boundary conditions and right-hand side are computed from the following exact solution:
\begin{align*}
    \uu(x,y) &= [y^2-2y+1+\nu(2x-1), x^2 - x - 2\nu(y-1)]^T; \\
    p(x,y) &= 2\nu (x+y-1)+\frac{1}{3k} - 4\nu^2; \\
    \psi(x,y) &= \frac{1}{k}(x(1-x)(y-1) + \frac{y^3}{3} - y^2 + y) + 2\nu x.
\end{align*}
We also set $k=\nu=3$ and $G=1$, in order for the parameters to satisfy the conditions of Lemma~\ref{lem:fov_main}.

As we have done for the Navier-Stokes problem in Section \ref{sec:ns} -- here, too, we provide a brief validation of our analysis. We again apply left preconditioning, using the diagonal and the upper-triangular preconditioners, $M_D$ and $M_U$ respectively, defined in \eqref{eq:MD} and \eqref{eq:MU}.

 Our observations are similar to those we made in Section \ref{sec:ns}. The results for a few mesh sizes can be found in Table \ref{tab:NSD}. We again observe an excellent level of scalability, with iteration counts nearly constant for various mesh sizes in all cases. The iteration counts in the $\ell_2$ and $H$ norms are nearly identical.

\begin{table}[htbp]
\centering

\begin{tabular}{|c|cc|cc|}
  \hline
  System Size & \multicolumn{2}{c|}{$\ell_2$-norm} & \multicolumn{2}{c|}{$H$-norm} \\
  \cline{2-5}
   & Diagonal & Upper Triangular & Diagonal & Upper Triangular \\
  \hline
  633 & 27 & 16 & 29 & 16 \\
  2,545 & 28 & 16 & 30 & 16 \\
  10,209 & 28 & 16 & 30 & 16 \\
  40,897 & 30 & 16 & 28 & 16 \\
  \hline
  \end{tabular}
\caption{Iteration counts for the Stokes-Darcy equations.  \chennew{The inner linear system solve was terminated when a relative residual tolerance of $\frac{\| r_k\| }{ \| r_0 \|} < 10^{-5}$ was reached,  where the norms used were the ones corresponding to results reported in the table: $\ell_2$-norm on the left and $H$-norm on the right.}}
\label{tab:NSD}
\end{table}
For the diagonal preconditioner, we have observed experimentally for the smaller-size problems that the parameters of Lemma \ref{lem:fov_main} satisfy $b \approx 9.18$, $c \approx 0.08$, and $bc<1$ \chen{in the $H$-norm}.
For the upper triangular preconditioner, $b \approx 8.28$, $c \approx 0.11$ and $bc<1$.

\section{Concluding Remarks}
\label{sec:conclusions}

Our analysis broadens the range of preconditioned saddle-point systems for which FOV analysis may be applied by including cases where zero is included in the field of values. This includes the important family of block-diagonal preconditioners, as well as upper-triangular preconditioners applied with right preconditioning. For these cases, to our knowledge, no FOV analysis was previously available when \eqref{eq:MAxx} is true.  

When applying Theorem~\ref{thm:fieldofvalues}, 
a disk must be excluded from the field of values, and the remaining part should not surround the origin, as we have illustrated in Figure~\ref{fig:aclt1}. \chen{To accomplish this, we require the imaginary part of the FOV to be small enough, which means that for the nonsymmetric saddle-point systems we consider,  the skew-symmetric part of \chennew{the} preconditioned operator needs to be small in norm. }

A finer geometric study of the field of values, beyond bounding it just by using the imaginary axis, may allow for broadening the scope of problems for which our analysis is applicable, including additional types nonsymmetric saddle-point linear systems.

\section*{Acknowledgments}\noindent We are very grateful to two  knowledgable referees whose  thorough and helpful reviews  have greatly  improved the quality of this paper.

\bibliographystyle{elsarticle-num}
\bibliography{references.bib}

\end{document}